\documentclass[12pt]{article}
\usepackage{amssymb,amsthm,amsmath,verbatim,enumerate,ifthen}
\usepackage{cite}
\usepackage[mathscr]{eucal}
\usepackage[utf8]{inputenc}
\usepackage[T1]{fontenc}
\usepackage{esint}
\usepackage{marginnote}
\newtheorem{thm}{Theorem}[section]
\newtheorem{cor}[thm]{Corollary}
\newtheorem{prop}[thm]{Proposition}
\newtheorem{lem}[thm]{Lemma}

\theoremstyle{definition}
\newtheorem*{defin}{Definition}

\theoremstyle{definition}

\newtheorem{exa}{Example}

\newtheorem*{keywords}{Keywords}
\newtheorem*{msc}{MSC}

\numberwithin{equation}{section}
\def\eq#1{{\rm(\ref{#1})}}
\def\Eq#1#2{\ifthenelse{\equal{#1}{*}}
  {\begin{equation*}\begin{aligned}[]#2\end{aligned}\end{equation*}}
  {\begin{equation}\begin{aligned}[]\label{#1}#2\end{aligned}\end{equation}}}

\def\A{\mathscr{A}}

\def\D{\mathscr{D}}
\def\E{\mathscr{E}}
\def\G{\mathscr{G}}
\def\M{\mathscr{M}}
\def\P{\mathscr{P}}
\newcommand\R{\mathbb{R}}

\newcommand\N{\mathbb{N}}
\newcommand\Z{\mathbb{Z}}

\newcommand\Q{\mathbb{Q}}
\def\WQ{\mathscr{Q}}
\def\WR{\mathscr{R}}

\newcommand{\operator}[1]{\mathop{\vphantom{\sum}\mathchoice
{\vcenter{\hbox{\LARGE $#1$}}}
{\vcenter{\hbox{\Large $#1$}}}{#1}{#1}}\displaylimits}

\def\Mm{\operator{\mathscr{M}}}
\def\Ar{\operator{\mathscr{A}}}       

\newcommand{\QA}[1]{\A_{#1}}

\DeclareMathOperator{\Quot}{Quot}
\DeclareMathOperator{\sign}{sign}

\newcommand{\dotvec}[3][SKIPPED]{
\ifthenelse{\equal{#1}{SKIPPED}}
  {#2,\dots,#3}
  {\underbrace{#2,\dots,#3}_{#1\text{ entries}}}
}
\title
{On Kedlaya type inequalities for weighted means}

\author{Zsolt P\'ales\footnote{{\it Corresponding author}, Institute of Mathematics, University of Debrecen, Pf.\ 12, 4010 Debrecen, Hungary. E-mail: {\tt pales@science.unideb.hu}}\,\, and Pawe{\l} Pasteczka\footnote{Institute of Mathematics, Pedagogical University of Cracow,  Podchor\k{a}\.{z}ych str~2, 30-084 Cracow, Poland. E-mail: {\tt pawel.pasteczka@up.krakow.pl}}}

\begin{document}
\maketitle

\begin{abstract}
In 2016 we proved that for every symmetric, repetition invariant and Jensen concave mean $\mathscr{M}$ the Kedlaya-type inequality
$$
\mathscr{A}\big(x_1,\mathscr{M}(x_1,x_2),\ldots,\mathscr{M}(x_1,\ldots,x_n)\big)
\le \mathscr{M} \big( x_1, \mathscr{A}(x_1,x_2),\ldots,\mathscr{A}(x_1,\ldots,x_n)\big)
$$
holds for an arbitrary $(x_n)$ ($\mathscr{A}$ stands for the arithmetic mean). We are going to prove the weighted counterpart of this inequality. More precisely, if $(x_n)$ is a vector with corresponding (non-normalized) weights $(\lambda_n)$ and $\mathscr{M}_{i=1}^n(x_i,\lambda_i)$ denotes the weighted mean then, under analogous conditions on $\mathscr{M}$, the inequality
$$
\mathscr{A}_{i=1}^n \big( \mathscr{M}_{j=1}^i (x_j,\lambda_j),\:\lambda_i\big) 
\le \mathscr{M}_{i=1}^n \big( \mathscr{A}_{j=1}^i (x_j,\lambda_j),\:\lambda_i\big)
$$
holds for every $(x_n)$ and $(\lambda_n)$ such that the sequence $(\frac{\lambda_k}{\lambda_1+\cdots+\lambda_k})$ is decreasing.
\end{abstract}

\begin{keywords}
Discrete mean; Weighted mean; Power mean; Quasi-arithmetic mean; Gini mean; Deviation mean; Jensen convexity, concavity; Kedlaya inequality.
\end{keywords}

\begin{msc}
Primary: 26D15, Secondary: 39B62.
\end{msc}

\section{Introduction}

In 1994 Kedlaya \cite{Ked94}, justifying Holland's conjecture \cite{Hol92}, proved that
\Eq{*}{
\frac{x_1+\sqrt{x_1 x_2}+\cdots+\sqrt[n]{x_1 x_2\cdots x_n}}{n} \le \sqrt[n]{x_1 \cdot \frac{x_1+x_2}{2} \cdots \frac{x_1+x_2+\cdots+x_n}{n}}
}
for every $x \in \R_+^n$ and $n\in\N$.

It motivated us to consider the following definition \cite{PalPas16}. Mean $\M \colon \bigcup_{n=1}^\infty I^n \to I$ ($I$ is an interval) is a \emph{Kedlaya mean} if (from now on $\A$ will denote arithmetic mean)
\Eq{DKI}{
\A\big(x_1,\M(x_1,x_2),\dots,\M(x_1,x_2,\dots,x_n)\big)
\le \M\big(x_1,\A(x_1,x_2),\dots,\A(x_1,x_2,\dots,x_n)\big) 
}
for every $n \in \N$ and $x \in I^n$. 

In this setting Kedlaya's result could be express briefly as {\it geometric mean is a Kedlaya mean}. Nevertheless, there appears a natural problem -- to find a broad family of Kedlaya means. For example, it is quite easy to prove that $\min$ and arithmetic means are Kedlaya means. Moreover convex combination of Kedlaya means are again a Kedlaya mean.

Some approach to this problem was given recently by authors in \cite{PalPas16}. We are going to present this result in a while, but we need to introduce some properties of means first.

Let $I\subseteq\R$ be an interval and let $\M \colon\bigcup_{n=1}^{\infty} I^n \to I$ be an arbitrary mean, i.e., for 
all $n\in\N$ and $(\dotvec{x_1}{x_n})\in I^n$, we assume that $\M$ satisfies the inequality 
\Eq{*}{
  \min(\dotvec{x_1}{x_n})\leq\M(\dotvec{x_1}{x_n})\leq\max(\dotvec{x_1}{x_n}).
}
We say that $\M$ is \emph{symmetric}, \emph{(strictly) increasing}, and \emph{Jensen convex (concave)} if, for all 
$n\in\N$, the $n$-variable restriction $\M|_{I^n}$ is a symmetric, (strictly) increasing in each of its variables, and 
Jensen convex (concave) on $I^n$, respectively.

A mean $\M$ is called \emph{repetition invariant} if, for all $n,m\in\N$
and $(\dotvec{x_1}{x_n})\in I^n$, the following identity is satisfied
\Eq{*}{
  \M(\dotvec{\dotvec[m]{x_1}{x_1}}{\dotvec[m]{x_n}{x_n}})
   =\M(\dotvec{x_1}{x_n}).
}

Having this in hand, let us recall one of the most important result from this paper.

\begin{thm}[\cite{PalPas16}, Theorem 2.1]
Every symmetric, Jensen concave and repetition invariant mean is a Kedlaya mean.
\end{thm}

As symmetry and repetition invariance are very natural axiom of means, Jensen concavity seamed to be the most restrictive one. Fortunately, it was characterized for many families of means. Many properties and characterizations are consequences of the general results obtained in a series of papers by Losonczi \cite{Los70a,Los71a,Los71b,Los71c,Los73a,Los77} (for Bajraktarevi\'c means and Gini means) and by Dar\'oczy \cite{Dar71b,Dar72b}, Dar\'oczy--Losonczi \cite{DarLos70}, Dar\'oczy--P\'ales \cite{DarPal82,DarPal83} (for deviation means) and by P\'ales \cite{Pal82a,Pal83b,Pal84a,Pal85a,Pal88a,Pal88d,Pal88e} (for deviation and quasi-deviation means), P\'ales and Pasteczka \cite{PalPas17a} (for quasi-arithmetic and homogeneous deviation means). Some results concerning Gaussian product were also given \cite{PalPas16}. It gives us plenty of examples of Kedlaya means.

Five years later in 1999 Kedlaya \cite{Ked99} improved his result to a weighted setting. In more details, he showed that

\begin{thm}[Kedlaya]
\label{thm:Kedlaya99}
Let $x_1,\dots,x_n,\lambda_1,\dots,\lambda_n$ be positive real numbers and define $\Lambda_k:=\lambda_1+\cdots+\lambda_k$. If the sequence $(\lambda_i/\Lambda_i)_{i=1}^n$ is nonincreasing then 
\Eq{*}{
\prod_{i=1}^n \Big( \sum_{j=1}^i \frac{\lambda_j}{\Lambda_i} x_j \Big)^{\lambda_i/\Lambda_n} \ge \sum_{j=1}^n \frac{\lambda_j}{\Lambda_n} \prod_{i=1}^j x_i^{\lambda_i/\Lambda_j}. 
}
\end{thm}

Motivated by these preliminaries, we are going to struggle with a weighted counterpart of Kedlaya inequality. Before it could be done we need to make some introduction to weighted means in abstract setting. We need to realize that there is no formal agreement concerning this definition. They were introduced for particular families only.

In this situation let us present weighted deviation and quasi-deviation means only. Formal definition of weighted means in the abstract setting will be introduced in the following section.

For an interval $I$, and a deviation function  $E \colon I^2 \to \R$ ($E(x,\cdot)$ is continuous and strictly increasing and $E(x,x)=0$, $x \in I$), for $x \in I^n$, we define a mean $y=\D_E(x)$ as a unique solution of equation
\Eq{E:Dev_def}{
\sum_{i=1}^n E(x_i,y)=0.
}
Its weighted counterpart is defined for any $x \in I^n$ and $\lambda \in \R_+^n$ as a unique solution of equation
\Eq{E:WDev_def}{
\sum_{i=1}^n \lambda_i \cdot E(x_i,y)=0.
}

This definition could be generalized further; if a function $E$ satisfying the following properties:
\begin{enumerate}[(a)]\itemsep=-2pt
 \item for all $(x,t) \in I^2$, $\sign E(x,t)=\sign(x-t)$;
 \item for all $x \in I$, $E(x,\cdot)$ is continuous;
 \item for all $x,y \in I$, the mapping $I \ni t \mapsto E(x,t)/E(y,t)$, $x<t<y$ is strictly increasing,
\end{enumerate}
then equalities \eqref{E:Dev_def} and \eqref{E:WDev_def} define the so-called \emph{quasi-deviation} and \emph{weighted quasi-deviation} means, respectively.

At the moment, each time we are dealing with a family which is a particular case of quasi-deviation means, weighted means are immediately defined. 
In this way, we can simply obtain quasi-arithmetic means, Gini means, Bajraktarevi\'c means etc. (cf. \cite{Bul03} for definitions) in their weighted setting.

Nevertheless, for the purpose of the present note, we need to separate the definition of weighted means from any particular family. This will be accomplished in the forthcoming section.

\section{Weighted means}

In this section we will introduce the notion of weighted means. Before we begin, let us underline few important facts. Weighted means are used very often among the literature. Most usually they are obtained by adding extra values to some symmetric operator (for example $\frac{\lambda_1 x_1+\cdots +\lambda_nx_n}{\lambda_1+\cdots+\lambda_n}$ instead of $\frac{x_1+\cdots+x_n}{n}$). It is done in this way that if we put $\lambda_1=\lambda_2=\dots=\lambda_n$ (very often weights are required to be normalized, that is $\sum \lambda_i=1$; see e.g. \cite{HarLitPol34}) then weighted mean goes back to non-weighted one. Due to this fact, whenever we say about weighted mean, its non-weighted counterpart is repetition invariant.

Let us also underline that in this definition weights are taken from some arbitrary ring $R \subset \R$. In fact, there are three particular rings which are significantly more important than any other: the ring of integers and the fields of rational numbers and real numbers.

As we will see, every repetition invariant mean generate (in a unique way) a weighted mean on rationals (roughly speaking it is implied by scaling invariance; see definition below). Reals are also of special interest, because each time we are dealing with quasi-deviation mean, we naturally request all real weights to be considered. 

\begin{defin}[Weighted means]
Let $I \subset \R$ be an arbitrary interval, $R \subset \R$ be a ring and, for 
$n\in\N$, define the set of $n$-dimensional weight vectors $W_n(R)$ by
\Eq{*}{
  W_n(R):=\{(\dotvec{\lambda_1}{\lambda_n})\in R^n\mid\lambda_1,\dots,\lambda_n\geq0,\,\lambda_1+\dots+\lambda_n>0\}.
}
\emph{A weighted mean on $I$ over $R$} or, in other words, \emph{an $R$-weighted mean on $I$} is a function 
\Eq{*}{
\M \colon \bigcup_{n=1}^{\infty} I^n \times W_n(R) \to I
}
satisfying the conditions (i)--(iv) presented below.
Elements belonging to $I$ will be called \emph{entries}; elements from $R$ -- \emph{weights}. 

\begin{enumerate}[(i)]
 \item \emph{Nullhomogeneity in the weights}: For all $n \in \N$, for all $(x,\lambda) \in I^n \times W_n(R)$, 
and $t \in R_+$,
 \Eq{*}{
   \M(x,\lambda)=\M(x,t \cdot \lambda),
 }
\item \emph{Reduction principle}: For all $n \in \N$ and for all $x \in I^n$, $\lambda,\mu \in W_n(R)$, 
\Eq{*}{
\M(x,\lambda+\mu)=\M(x\odot x,\lambda\odot\mu),
}
where $\odot$ is a \emph{shuffle operator}\footnote{This definition comes from the theory of computation. Perhaps the most famous (folk) result states that shuffling of two regular languages is again regular; see e.g. \cite{BloEsi97}.} defined as
\Eq{*}{
(\dotvec{p_1}{p_n})\odot (\dotvec{q_1}{q_n}):=(\dotvec{p_1,q_1}{p_n,q_n}).
}
\item \emph{Mean value property}: For all $n \in \N$, for all $(x,\lambda) \in I^n \times W_n(R)$
\Eq{*}{
\min(\dotvec{x_1}{x_n}) \le \M(x,\lambda)\le \max(\dotvec{x_1}{x_n}),
}
\item \emph{Elimination principle}: For all $n \in \N$, for all $(x,\lambda) \in I^n \times W_n(R)$ and for 
all $j\in\{1,\dots,n\}$ such that $\lambda_j =0$,
\Eq{*}{
\M(x,\lambda) = \M\big((x_i)_{i\in\{1,\dots,n\}\setminus\{j\}},(\lambda_i)_{i\in\{1,\dots,n\}\setminus\{j\}}\big),
}
i.e., entries with a zero weight can be omitted. 
\end{enumerate}
\end{defin}

For the sake of convenience, we will use the sum-type abbreviation
\Eq{*}{
\Mm_{i=1}^n(x_i,\lambda_i):=\M\big((\dotvec{x_1}{x_n}),(\lambda_1,\dots,\lambda_n)\big).
}

Let us begin with some technical lemma. To avoid misunderstanding, if we have a finite sequence $(a_1,\dots,a_n)$ and $k,m\in\{1,\dots,n\}$
such that $k<m$, then $(a_m,\dots,a_k)$ will be interpreted as the empty sequence. 

\begin{lem}
\label{lem:equiv_red}
Let $I$ be an arbitrary interval, $R \subset \R$ be a ring, $\M$ be a weighted mean defined on $I$ over $R$.  
For every $n \in \N$, $k\in\{1,\dots,n\}$, $x \in I^n$, $\lambda \in W_n(R)$ and a nonnegative number $\lambda_k'\in R$, we have
\Eq{*}{
  \M\big(&(x_1,\dots,x_{k-1},x_k,x_k,x_{k+1},\dots,x_n),\:(\lambda_1,\dots,\lambda_{k-1},\lambda_k,\lambda_k',\lambda_{k+1},\dots,\lambda_n)\big)\\
 &=\M\big((x_1,\dots,x_{k-1},x_k,x_{k+1},\dots,x_n),\:(\lambda_1,\dots,\lambda_{k-1},\lambda_k+\lambda_k',\lambda_{k+1},\dots,\lambda_n)\big)
}
\end{lem}
\begin{proof}
If $\lambda_k'=0$, then the statement follows from the elimination principle immediately. In the other case, for $i\in\{1,\dots,n\}$,
define $\lambda_i':=\delta_{ik}\lambda_k'$, where $\delta$ stands for the Kronecker symbol.
Applying the elimination principle iteratively $n-1$ times, and then using the reduction principle, we obtain
\Eq{*}{
\M\big((x_1,\dots,x_{k-1},x_k,x_k,x_{k+1},\dots,x_n),\:(\lambda_1,\dots,\lambda_{k-1},&\lambda_k,\lambda_k',\lambda_{k+1},\dots,\lambda_n)\big)\\
 &=\M(x\odot x,\lambda\odot\lambda')=\M(x,\lambda+\lambda'),
}
which is exactly the identitity to be proved.
\end{proof}

In the following theorem we will prove that a weighted mean defined on a ring can be extended to its quotient field denoted as $\Quot(R)$. 

\begin{thm}
\label{thm:fieldex}
Let $I$ be an interval, $R \subset \R$ be a ring, $\M$ be a weighted mean defined on $I$ over $R$.
Then there exists a unique mean $\widetilde\M$ defined on $I$ over $\Quot(R)$ such that 
\Eq{*}{
\widetilde \M \vert_{\bigcup_{n=1}^{+\infty} I^n \times W_n(R)} =\M.
}
Moreover if $\M$ is symmetric/monotone then so is $\widetilde\M$.
\end{thm}

\begin{proof}
Fix $n \in \N$, $x \in I^n$, and $\lambda \in W_n(\Quot(R))$. Then there exists $q \in R$ such that $q\lambda \in W_n(R)$ (for example a product of all denominators). We define
\Eq{E:deftildeM}{
\widetilde \M (x,\lambda):=\M(x,q\lambda).
}
To prove the correctness of this definition, it suffices to show that it does not depend on the selection of $q$. Indeed, take $q' \in R$ such that $q'\lambda \in W_n(R)$.
We need to verify if the equality $\M(x,q\lambda)=\M(x,q'\lambda)$ is valid. However, applying the nullhomogeneity of $\M$ (twice), we get
\Eq{*}{
\M(x,q\lambda)=\M(x,q'q\lambda)=\M(x,qq'\lambda)=\M(x,q'\lambda).
}

In order to verify the nullhomogeneity of $\widetilde \M$, observe that every positive element of $\Quot(R)$ can be 
represented as $a/b$ for some $a,\,b \in R_+$. Then, obviously, $ bq \cdot (a/b) \cdot \lambda \in W_n(R)$. Thus
\Eq{*}{
\widetilde \M(x,(a/b) \cdot \lambda)=\M(x,bq \cdot (a/b) \cdot \lambda)=\M\big(x,a \cdot (q\lambda)\big)=\M(x,q\lambda)= \widetilde\M(x,\lambda).
}

To prove reduction principle, take $\lambda,\,\mu \in W_n(\Quot(R))$ arbitrarily. Then there exist $q,\,r \in R$ such that $q \lambda, r\mu \in W_n(R)$. In this case we also have $qr\lambda,\,qr\mu \in W_n(R)$. Then $(qr\lambda) \odot (qr\mu) \in W_{2n}(R)$ and $(qr\lambda)\odot(qr\mu)=qr\cdot(\lambda\odot\mu)$. Having these properties, we obtain
\Eq{*}{
\widetilde \M(x \odot x,\lambda \odot \mu)&=\M \big(x \odot x,qr\cdot(\lambda\odot\mu)\big)\\
 &=\M\big(x \odot x,(qr\lambda)\odot(qr\mu)\big)
 =\M\big(x,qr\lambda+qr\mu\big)=\widetilde\M(x,\lambda+\mu).
}

The two remaining properties (mean value property, elimination principle) are obvious. Moreover part is simply implied 
by \eq{E:deftildeM}.
\end{proof}

What we are going to prove now is that every repetition invariant (non-weighted) mean can be associated with a  
$\Z$-weighted and, in the virtue of Theorem~\ref{thm:fieldex}, a $\Q$-weighted mean. In fact this operation can be also 
reversed.

\begin{thm}\label{thm:weighted_nonweighted}
If $\M\colon\bigcup_{n=1}^\infty I^n\to I$ is a repetition invariant mean on $I$, then the formula  
\Eq{WeiQDef}{
\widetilde\M \big((\dotvec{x_1}{x_n}),(\lambda_1,\dots,\lambda_n)\big):=\M\big( 
\dotvec[\lambda_1]{x_1}{x_1},\dots,\dotvec[\lambda_n]{x_n}{x_n}\big)
}
defines a weighted mean $\widetilde\M\colon\bigcup_{n=1}^\infty I^n\times W_n(\Z) \to I$ on $I$ over $\Z$.

Conversely, if $\widetilde\M\colon\bigcup_{n=1}^\infty I^n\times W_n(\Z) \to I$ is a $\Z$-weighted mean on $I$, then 
\Eq{E:MMtilde}{
\M(\dotvec{x_1}{x_n}):=\widetilde \M\big((\dotvec{x_1}{x_n}),(\dotvec[n]11)\big)
}
is a repetition invariant mean on $I$. Furthermore these transformations are inverses of each other.
\end{thm}

\begin{proof}
Clearly, the transformations described in the theorem are inverses of each other.

Let $\M$ be a repetition invariant mean on $I$ and let $\widetilde\M$ be given by \eq{WeiQDef}.
We need to show that $\widetilde\M$ satisfies all properties (i)--(iv) listed in the definition of weighted means. 
First observe that $\widetilde\M$ obviously admits the mean value property. Elimination principle is also immediate 
because if $\lambda_j=0$ then element $x_j$ does not appear on the right hand side of \eq{WeiQDef}.

Let us now verify the nullhomogeneity in the weights. For $t \in \N_+$, we can apply repetition invariance of $\M$ to get,
\Eq{*}{
\widetilde\M \big((\dotvec{x_1}{x_n}),&(t\lambda_1,\dots,t\lambda_n)\big) 
=\M\big( \dotvec[t\cdot \lambda_1]{x_1}{x_1},\dots,\dotvec[t\cdot \lambda_n]{x_n}{x_n}\big)\\
&=\M\big( \dotvec[\lambda_1]{x_1}{x_1},\dots,\dotvec[\lambda_n]{x_n}{x_n}\big)
=\widetilde \M \big((\dotvec{x_1}{x_n}),(\lambda_1,\dots,\lambda_n)\big).
}

Finally, we will prove the reduction principle. We may assume that $\lambda,\,\mu \in \N^n$. Then, for all $x\in I^n$,
\Eq{*}{
\widetilde \M (x,\lambda+\mu)
&=\M\big( \dotvec[\lambda_1+\mu_1]{x_1}{x_1},\dots,\dotvec[\lambda_n+\mu_n]{x_n}{x_n}\big)\\
&=\M\big( \dotvec[\lambda_1]{x_1}{x_1},\dotvec[\mu_1]{x_1}{x_1},\dots,\dotvec[\lambda_n]{x_n}{x_n},\dotvec[\mu_n]{x_1}{x_1}\big)\\
&=\widetilde \M \big((\dotvec{x_1,x_1}{x_n,x_n}),(\dotvec{\lambda_1,\mu_1}{\lambda_n,\mu_n})\big)
=\widetilde \M (x\odot x,\lambda\odot\mu).
}

Now we will prove the converse part. Let $\widetilde\M$ be a $\Z$-weighted mean on $I$. By the definition, we get
\Eq{*}{
\M(\dotvec{x_1}{x_n})=\widetilde \M\big((\dotvec{x_1}{x_n}),(\dotvec11)\big) \le \max(\dotvec{x_1}{x_n});
}
similarly $\M(\dotvec{x_1}{x_n}) \ge \min(\dotvec{x_1}{x_n})$. 

To prove the repetition invariance of $\M$, take any $m \in \N$. 
By the definition $\widetilde\M$ and the nullhomogeneity, this property is equivalent to
\Eq{*}{
  \widetilde\M\big((\dotvec[m]{x_1}{x_1},\dots,\dotvec[m]{x_n}{x_n}),(\dotvec[mn]{1}{1})\big)
   =\widetilde\M\big((\dotvec{x_1}{x_n}),(\dotvec[n]mm)\big).
}
To see this equality, we shall apply Lemma~\ref{lem:equiv_red} iteratively to encompass each block appearing on the left hand side.
\end{proof}

Let us now introduce some natural properties of weighted means. A weighted mean $\M\colon\bigcup_{n=1}^\infty I^n\times W_n(R) \to I$ is said to be \emph{symmetric}, if for all $n \in \N$, $x \in I^n$, $\lambda \in W_n(R)$, and for all permutations $\sigma\in S_n$, 
\Eq{*}{
\M(x,\lambda) =\M(x\circ\sigma,\lambda\circ\sigma).
}
We will call a weighted mean $\M$ \emph{Jensen concave} if, for all $n \in \N$, $x,y \in I^n$ and $\lambda \in W_n(R)$,
\Eq{E:JF2}{
\M \Big( \frac{x+y}2 , \lambda \Big) \ge\frac12 \big( \M(x,\lambda)+\M(y,\lambda) \big).
}
If, on the above indicated domain, the reversed inequality is satisfied, then $\M$ is said to be Jensen convex. First 
observe that, given a (symmetric) Jensen concave mean $R$ weighted mean $\M\colon\bigcup_{n=1}^\infty I^n\times W_n(R) 
\to I$, the mean $\widehat\M\colon\bigcup_{n=1}^\infty (-I)^n\times W_n(R) \to(-I)$ defined by
\Eq{hat}{
  \widehat\M(x,\lambda):=-\M(-x,\lambda) \qquad(n\in\N,\,x\in(-I)^n,\,\lambda\in W_n(R))
}
is a (symmetric) Jensen convex $R$-weighted mean on $(-I)$. Therefore, everything that we obtain in terms of Jensen concavity, can be rewritten for Jensen convexity, and vice versa.

Another important observation is that, due to the mean value property, means are locally bounded functions. Therefore, as a consequence of the celebrated Bernstein--Doetsch Theorem (cf.\ \cite{BerDoe15}, \cite{Kuc85}), Jensen concavity or Jensen convexity is equivalent to their concavity or convexity, respectively. Henceforth, it implies their continuity with respect to their entries over the interior of $I^n$. 

A weighted mean $\M$ is said to be \emph{continuous in the weights} if, for all $n \in \N$ and $x \in I^n$, the mapping $\lambda \mapsto \M(x,\lambda)$ is continuous on $W_n(R)$.

The following two statements are easy to see.

\begin{thm}\label{thm:sym}
If $\M$ is a symmetric repetition invariant mean on $I$, then the function $\widetilde\M$ defined by the formula \eq{WeiQDef}
is a symmetric weighted mean on $I$ over $\Z$. 

Conversely, if $\widetilde \M$ is a symmetric $\Z$-weighted mean on $I$, then the function $\M$ defined by 
\eq{E:MMtilde} is a symmetric repetition invariant mean on $I$. 
\end{thm}

\begin{thm}\label{thm:conc}
If $\M$ is a Jensen concave repetition invariant mean on $I$, then the function $\widetilde\M$ defined by the formula \eq{WeiQDef}
is a Jensen concave weighted mean on $I$ over $\Z$. 

Conversely, if $\widetilde \M$ is a Jensen concave $\Z$-weighted mean on $I$, then the function $\M$ defined by 
\eq{E:MMtilde} is a Jensen concave repetition invariant mean on $I$. 
\end{thm}

Usually, instead of explicitly writing down weights, we can consider a function with finite range as the argument of the 
given mean. Let $R$ be a subring of $\R$. We say that $D\subseteq\R$ is an \emph{$R$-interval} if $D$ is of the form 
$[a,b)$, where $a,b\in R$. The Cartesian product of two $R$-intervals will be called an $R$-rectangle. The length of an 
interval $D$ will be denoted by $|D|$.

Given an $R$-interval $D$, a function $f\colon D\to I$ is called \emph{$R$-simple} if there exist $n \in \N$ and a 
partition of $D$ into $R$-intervals $\{D_i\}_{i=1}^{n}$ such that $\sup D_i=\inf D_{i+1}$ for $i\in\{1,\dots,n-1\}$ and 
$f$ is constant on each subinterval $D_i$. Then, for an $R$-weighted mean $\M$ on $I$, we set
\Eq{*}{
\Mm f(x)dx:=\Mm_{i=1}^n (f|_{D_i},|D_i|)=\M((f|_{D_1},\dots,f|_{D_n}),(|D_1|,\dots,|D_n|)).
}

Given an $R$-rectangle $D\times E$, a function $f \colon D \times E \to I$ is called \emph{$R$-simple} if there exists 
$n\in\N$ and a partition of $D \times E$ into $R$-rectangles $\{D_i\times E_i\}_{i=1}^n$ such that $f$ is constant on 
every $D_i\times E_i$. One can easily see that, for every $x\in D$, $y\in E$, the mappings $f(x,\cdot)$ and $f(\cdot,y)$ 
are $R$-simple functions on $E$ and $D$, respectively. 

A subset $H\subseteq \R$ or $H\subseteq \R^2$ will be called \emph{$R$-simple} if its characteristic function is 
$R$-simple. It is easy to see that a set $H$ is $R$-simple if and only is it is the disjoint union of finitely many 
$R$-intervals or $R$-rectangles, respectively. 

For an $R$-simple set $H\subseteq\R$, the sum of the lengths of the decomposing $R$-intervals will be denoted by $|H|$. 
In fact, this is the Lebesgue measure of $H$.

In this section we will prove two important lemmas

\begin{lem}\label{L:theta}
Let $R$ be a ring such that $\Q R\subseteq R$.
Then, for every $R$-rectangle $D\times E$ and $\theta \in \Q \cap [0,1]$, there exists a $R$-simple subset $H\subseteq D\times E$ such that
\begin{enumerate} 
 \item for all $x \in D$, $|\{y \colon (x,y)\in H\}|=\theta \cdot |E|$,
 \item for all $y \in E$, $|\{x \colon (x,y)\in H\}|=\theta \cdot |D|$.
\end{enumerate}
A set $H$ with the above properties will be called a $\theta$-proportional subset of $D\times E$.
\end{lem}

\begin{proof} Let $D$ and $E$ be arbitrary $R$-intervals. Let us recall first that there exits an affine bijection $\varphi \colon [0,1)^2\to D \times E$. If $D=[a,b)$ and $E=[c,d)$, then such an affine bijection can be given by
\Eq{*}{
  \varphi(t,s):=((1-t)a+tb,(1-s)c+sd)\qquad((t,s)\in[0,1)^2).
}

Assume that $\theta$ is of the form $p/q$, where $q\in \N$, $p\in\{0,\dots,q\}$. Now set
\Eq{*}{
H_{i,j}:=\left[\frac iq,\frac {i+1}q\right) \times \left[\frac jq,\frac {j+1}q\right),\qquad i,\,j\in\{\dotvec0{q-1}\}.
}
Finally, define the set $H_0\subseteq[0,1)^2$ by
\Eq{*}{
H_0:=\bigcup_{i=0}^{q-1} \bigcup_{j=i}^{i+p-1} H_{i,j(\text{mod } q)}.
}
It is simple to verify that $H_0$ is a $\theta$-proportional subset of $[0,1)^2$. Therefore, the set $H:=\varphi(H_0)$ is $\theta$-proportional subset of $D\times E$.
\end{proof}

The inequality stated in the next result will be called the Jensen--Fubini inequality in the sequel. We remind the 
reader that the symbol $\A$ stands for the arithemetic mean.

\begin{lem}
Let $D$ and $E$ be $\Q$-intervals. Let $\M\colon\bigcup_{n=1}^\infty I^n\times W_n(\Q) \to I$ be a $\Q$-weighted mean on $I$. Then, $\M$ is Jensen concave if and only if, for every $\Q$-simple function $f \colon D \times E \to I$, we have
\Eq{E:JF}{
  \Ar \Big(\Mm f(x,y)\:dy\Big)dx \leq \Mm \Big(\Ar f(x,y)\:dx\Big)dy .
}
In addition, the validity of the reversed inequality in \eq{E:JF} characterizes the Jensen convexity of $\M$.
\end{lem}

\begin{proof}
Assume first that $\M$ is Jensen concave. Let $f \colon D \times E \to I$ be a $\Q$-simple function.
Then $D\times E$ can be partitioned into a finite number of $\Q$-rectangles $\{D_i\times E_i\}_{i=1}^N$ such that $f \vert_{D_i\times E_i}$ is constant.  Therefore there exists a number $M \in \N$ (being a product of all denominators of the endpoints of $D_i$ and $E_i$) such that $M\cdot D_i$ and $M\cdot E_i$ are $\Z$-intervals for all $i \in \{1,\dots,N\}$.

Having this, we can stretch $f$ to a $\Z$-simple function $\widetilde f \colon (M \cdot D) \times (M \cdot E)\to I$ defined by
\Eq{*}{
\widetilde f (x,y):=f(x/M,y/M).
}
On the other hand, the nullhomogeneity of $\M$ and also of $\A$ in the weights implies
\Eq{*}{
&\Mm \Big( \Ar f(x,y)\:dx\Big)dy=\Mm \Big( \Ar \widetilde f(x,y)\:dx\Big)dy\\
\text{ and }\quad&\Ar \Big( \Mm f(x,y)\:dy\Big)dx=\Ar \Big(\Mm \widetilde f(x,y)\:dy\Big)dx.
}
Therefore we may assume that initial function $f$ is $\Z$-simple and $D$, $E$ are $\Z$-intervals.
Furthermore (just to make the notation simple) we can shift the left-bottom corner of $D \times E$ to the origin, that is we assume that $D=[0,n)$, $E=[0,m)$ for some $m,n \in\N$. Then we can construct a matrix $(a_{i,j})_{\substack{i \in \{1,\dots,n\} \\ j \in \{1,\dots,m\}}}$ with entries in $I$ such that 
\Eq{*}{
\widetilde f(x,y)=a_{i,j}\text{ for }(x,y) \in [i-1,i)\times [j-1,j)\text{, where }i \in \{1,\dots,n\},\: j \in \{1,\dots,m\}.
}
Then we have
\Eq{*}{
\Mm \Ar \widetilde f(x,y)\:dx\:dy &=\Mm_{j=1}^m  \left(\frac{a_{1,j}+\cdots+a_{n,j}}{n},\,1 \right), \\
\Ar \Mm \widetilde f(x,y)\:dy\:dx &=\frac 1n \sum_{i=1}^n \Mm_{j=1}^m (a_{i,j},\,1).
}
Finally, applying the Jensen concavity of $\M$, we obtain the following inequality
\Eq{*}{
  \frac 1n \sum_{i=1}^n \Mm_{j=1}^m (a_{i,j},\,1)
  \leq\Mm_{j=1}^m  \left(\frac{a_{1,j}+\cdots+a_{n,j}}{n},\,1 \right)
}
which implies \eq{E:JF}.

To complete the proof, assume that \eq{E:JF} holds for all $\Q$-simple function $f \colon D \times E \to I$. To prove the Jensen concavity of the mean $\M$, let $x,y\in I^n$ and $\lambda\in W_n(\Q)$. We may assume that $\lambda_i>0$ for all $i\in\{1,\dots,n\}$. Let $E$ be a $\Q$-interval which is partitioned into some $\Q$-intervals $\{E_i\}_{i=1}^n$ such that $|E_i|=\lambda_i$ for all $i\in\{1,\dots,n\}$. Now construct the function $f\colon [0,2)\times E\to I$ as follows:
\Eq{*}{
  f(u,v)=\begin{cases}
         x_i &\mbox{if } u\in [0,1),\, v\in E_i,\\ 
         y_i &\mbox{if } u\in [1,2),\, v\in E_i.
         \end{cases}
}
Then, obviously, $f$ is a $\Q$-simple function. Applying \eq{E:JF} for this $f$, it follows that
\Eq{*}{
  \frac12 \big( \M(x,\lambda)+\M(y,\lambda) \big)
  =\Ar \Big(\Mm f(u,v)\:dv\Big)du 
  \leq \Mm \Big(\Ar f(u,v)\:du\Big)dv 
  =\M \Big( \frac{x+y}2 , \lambda \Big),
}
which shows that $\M$ is Jensen concave, indeed.

The last assertion of the theorem can be obtained by the transformation $\M\mapsto\widehat\M$ defined in \eq{hat}.
\end{proof}

\section{Results: The weighted Kedlaya inequality}

We are heading toward the inequality which is main target for the present paper. 

To have a weighed counterpart of the Kedlaya inequality, we have to take weight sequences $\lambda$ from $R$ with a positive first member.
Therefore, for a given ring $R$, we define
\Eq{*}{
  W^0_n(R)&:=\{(\lambda_1,\dots,\lambda_n)\in R^n\mid \lambda_1>0,\,\lambda_2,\dots,\lambda_n\geq0\}, \qquad (n \in \N),\\
  W^0(R)&:=\{\lambda\in R^\N\mid \lambda_1>0,\,\lambda_2,\lambda_3,\dots\geq0\}.
}
The nonincreasingness of the ratio sequence $\big(\tfrac{\lambda_i}{\lambda_1+\cdots+\lambda_i}\big)$ will be a key assumption for Kedlaya type inequalities, therefore, we also set
\Eq{*}{
  V_n(R)&:=\big\{\lambda\in W_n^0(R)\mid \big(\tfrac{\lambda_i}{\lambda_1+\cdots+\lambda_i}\big)_{i=1}^n \mbox{ is nonincreasing}\big\}, 
  \qquad (n \in \N),\\
   V(R)&:=\big\{\lambda\in W^0(R)\mid \big(\tfrac{\lambda_i}{\lambda_1+\cdots+\lambda_i}\big)_{i=1}^\infty \mbox{ is nonincreasing}\big\}.
}

Given $n\in\N$ and a weight sequence $\lambda \in W_n^0(R)$, we say that a weighted mean $\M \colon  \bigcup_{n=1}^{\infty} I^n \times W_n(R) \to I$ satisfies the \emph{$n$ variable $\lambda$-weighted Kedlaya inequality}, or shortly, the \emph{$(n,\lambda)$-Kedlaya inequality} if 
\Eq{KI}{
\Ar_{k=1}^n \left( \Mm_{i=1}^k (x_i,\lambda_i),\:\lambda_k\right) \le \Mm_{k=1}^n \left( \Ar_{i=1}^k (x_i,\lambda_i),\:\lambda_k\right)
  \qquad(x\in I^n).
}
If $\lambda \in W^0(R)$ and this inequality holds for all $n\in\N$, then we say that $\M$ satisfies the \emph{$\lambda$-weighted Kedlaya inequality}, or shortly, the \emph{$\lambda$-Kedlaya inequality}. The main result of the present note is to provide a sufficient condition for the weight sequence $\lambda$ and the weighted mean $\M$ such that the $n$ variable $\lambda$-weighted Kedlaya inequality is satisfied by $\M$.  

\begin{thm} 
\label{thm:main}
Let $n \in \N$, $\lambda \in V_n(\Q)$ and let $\M\colon\bigcup_{n=1}^\infty I^n\times W_n(\Q) \to I$ be a symmetric and Jensen concave $\Q$-weighted mean on $I$. Then $\M$ satisfies the $n$ variable $\lambda$-weighted Kedlaya inequality \eq{KI}.

On the other hand, if $\M$ is a symmetric and Jensen convex $\Q$-weighted mean on $I$, then \eq{KI} holds with reversed inequality. 
\end{thm}
\begin{proof}
The statement of the theorem is trivial if $n=1$. Therefore, we may assume that $n\geq2$. 
Denote, for $k\in\{1,\dots,n\}$, the partial sum $\lambda_1+\dots+\lambda_k$ by $\Lambda_k$ and set $\Lambda_0:=0$.

First observe that if $\lambda_i=0$ for some $i\in \{2,\dots,n\}$, then, for all $j\in \{i,\dots,n\}$, we get $\lambda_j/\Lambda_j\le \lambda_i/\Lambda_i=0$, that is $\lambda_j=0$ for all $j\in \{i,\dots,n\}$ and, consequently, the $n$ variable Kedlaya inequality is equivalent to the $(i-1)$ variable Kedlaya inequality. Thus, from now on we assume that $\lambda_i>0$ for all $i \in \{1,\dots,n\}$.

Take an arbitrary vector $x \in I^n$ and, for $k \in \{1,\dots,n\}$, denote 
\Eq{*}{
m_k:= \Ar_{i=1}^k(x_i,\lambda_i)=\frac{\lambda_1x_1+\dots+\lambda_kx_k}{\Lambda_k}.
}
In what follows, we are going to prove that, for all $j \in \{2,\dots,n\}$,
\Eq{E:IND0}{
\Lambda_{j-1} \cdot \Mm_{i=1}^{j-1}(m_i,\lambda_i) + \lambda_j \cdot \Mm_{i=1}^{j}(x_i,\lambda_i) \le\Lambda_j \cdot \Mm_{i=1}^{j}(m_i,\lambda_i).
}
Then, applying this inequality for all $j \in \{2,\dots,n\}$, summing up side by side, after simple reduction, we get
\Eq{*}{
\sum_{j=1}^n \lambda_j \cdot \Mm_{i=1}^{j}(x_i,\lambda_i) \le\Lambda_n \cdot \Mm_{i=1}^{n}(m_i,\lambda_i).
}
Then, after dividing both sides of this inequality by $\Lambda_n$, we arrive at \eq{KI}.
For the sake of convenience let us rewrite \eq{E:IND0} into the following equivalent form
\Eq{E:IND}{
\frac{\Lambda_{j-1}}{\Lambda_j} \cdot \Mm_{i=1}^{j-1}(m_i,\lambda_i) + \frac{\lambda_j}{\Lambda_j} \cdot \Mm_{i=1}^{j}(x_i,\lambda_i) \le \Mm_{i=1}^{j}(m_i,\lambda_i).
}
To prove this, we will define a $\Q$-simple function $f: [0,\Lambda_j)^2 \to \R_+$ such that respective sides of the inequality \eqref{E:JF} and \eqref{E:IND} coincide with each other. This will complete the proof of this theorem.

Consider a partition of the domain of $f$ into the blocks $B_k:=[0,\Lambda_{j-1})\times [\Lambda_{k-1},\Lambda_k)$ and $C_k:=[\Lambda_{j-1},\Lambda_j)\times [\Lambda_{k-1},\Lambda_k)$, where $k\in\{1,\dots,j\}$. Now, based on Lemma~\ref{L:theta}, let $H_k$ be a fixed $\frac{\lambda_j \Lambda_{k-1}}{\lambda_k \Lambda_{j-1}}$-proportional subset of the block $B_k$ for all $k\in\{1,\dots,j\}$ and define
\begin{align*}
f(x,y)&:= \begin{cases} 
    m_{k-1} &\text{ for } (x,y) \in H_k,\quad k=2,\dots,j; \\
    m_{k} &\text{ for } (x,y) \in B_k \setminus H_k,\quad k=1,\dots,j-1; \\
    x_k &\text{ for } (x,y) \in C_k,\quad k=1,\dots,j.
       \end{cases}
 \end{align*}
 
To verify the correctness of this definition, we need to check  $\frac{\lambda_j \Lambda_{k-1}}{\lambda_k \Lambda_{j-1}}\leq1$ for $k\in\{1,\dots,j\}$. An elementary calculation shows that this inequality holds if and only if  $\lambda_k /\Lambda_k \ge \lambda_j / \Lambda_j$, what is provided by the assumption on the weight vector $\lambda$.

Fix $x_0 \in [0,\Lambda_{j-1})$. By the construction of $f$, we have that, for $k\in\{1,\dots,j-1\}$, $f(x_0,y)=m_k$ if $(x_0,y) \in (B_k \setminus H_k)\cup H_{k+1}$. On the other hand,
\begin{align*}
 |\big\{y \colon (x_0,y) \in (B_k \setminus H_k) \cup H_{k+1}\big\}|
 &= \lambda_k\bigg(1-\frac{\lambda_j \Lambda_{k-1}}{\lambda_k \Lambda_{j-1}}\bigg)+\lambda_{k+1} \frac{\lambda_j \Lambda_k}{\lambda_{k+1} \Lambda_{j-1}}\\
 &=\frac{\lambda_k \Lambda_{j-1}-\lambda_j \Lambda_{k-1}+\lambda_j \Lambda_k}{\Lambda_{j-1}} 
 =\frac{\lambda_k \Lambda_{j-1}+\lambda_j \lambda_k}{\Lambda_{j-1}} 
 =\frac{\Lambda_j}{\Lambda_{j-1}} \lambda_k.
\end{align*}
Then, by the symmetry of $\M$ and the definition of the $\M$-integral, for all $x_0 \in [0,\Lambda_{j-1})$, we have
\Eq{*}{
\Mm f(x_0,y)\:dy
=\Mm_{k=1}^{j-1}\Big(m_k,\frac{\Lambda_j}{\Lambda_{j-1}} \lambda_k \Big)
=\Mm_{k=1}^{j-1}(m_k,\lambda_k).
}
For $x_0 \in [\Lambda_{j-1},\Lambda_j)$, we simply get 
\Eq{*}{
\Mm f(x_0,y)\:dy=\Mm_{k=1}^{j}(x_k,\lambda_k ).
}
We can now calculate the weighted arithmetic mean with respect to $x$ and obtain
\Eq{*}{
\Ar\Big(\Mm f(x,y)\:dy\Big)dx
=\frac{\Lambda_{j-1}}{\Lambda_j} \cdot\Mm_{k=1}^{j-1}(m_k,\lambda_k) + \frac{\lambda_j}{\Lambda_j}  \cdot \Mm_{k=1}^{j}(x_k,\lambda_k).
}
This proves that the left hand sides of \eqref{E:IND} and \eqref{E:JF} are equal to each other. 

Finally, we shall prove that it is also the case for the right hand sides. For, it suffices to prove that
\Eq{E2}{
\Ar f(x,y_0)\:dx=m_i,\qquad y_0 \in [\Lambda_{i-1},\Lambda_i),\quad i\in\{1,\dots,j\}.
}

For $y_0 \in [\Lambda_{0},\Lambda_1)$, this equality is the consequence of the trivial equality $m_1=x_1$.
For $k\in\{2,\dots,j\}$ and $y_0 \in [\Lambda_{k-1},\Lambda_k)$, we have that $f(x,y_0)$ equals $m_{k-1}$, $m_k$, or $x_k$ on $H_k$, $B_k \setminus H_k$, and $C_k$, respectively. But by the proportionality property of $H_k$, we know that
\Eq{*}{
|\{x\colon (x,y_0)\in H_k\}|=\frac{\lambda_j \Lambda_{k-1}}{\lambda_k \Lambda_{j-1}} \cdot |\{x\colon (x,y_0)\in B_k\}|
=\frac{\lambda_j \Lambda_{k-1}}{\lambda_k}.
}
Therefore,
\Eq{*}{
|\{x\colon (x,y_0)\in B_k \setminus H_k\}|=\Lambda_{j-1}-\frac{\lambda_j \Lambda_{k-1}}{\lambda_k},
}
and we also have
\Eq{*}{
|\{x\colon (x,y_0)\in C_k\}|=\lambda_j.
}
Obviously the total length of the slice $\{x\colon (x,y_0)\in B_k \cup C_k\}$ equals $\Lambda_j$. Using this and the easy-to-see identity $x_k=(\Lambda_km_k-\Lambda_{k-1}m_{k-1})/\lambda_k$, we get
\Eq{*}{
\Ar f(x,y_0)\:dx&=\frac{1}{\Lambda_j} \bigg( \frac{\lambda_j\Lambda_{k-1}}{\lambda_k} m_{k-1} + \Big( \Lambda_{j-1}- \frac{\lambda_j\Lambda_{k-1}}{\lambda_k} \Big) m_k + \lambda_j \cdot \frac{\Lambda_km_k-\Lambda_{k-1}m_{k-1}}{\lambda_k}\bigg)\\
&=\frac{m_k}{\Lambda_j\lambda_k} \big( \Lambda_{j-1}\lambda_k- \lambda_j \Lambda_{k-1}+\lambda_j\Lambda_k\big)
=\frac{m_k}{\Lambda_j\lambda_k} \big(\Lambda_{j-1}\lambda_k+ \lambda_j \lambda_k\big)=m_k.
}
Therefore, the corresponding sides of \eq{E:IND} and \eq{E:JF} coincide. As the Jensen concavity of $\M$ implies the Jensen--Fubini inequality \eq{E:JF}, we obtain \eq{E:IND}, and hence \eq{E:IND0} and, finally, the desired inequality \eq{KI}.

The last assertion of the theorem can be obtained by the transformation $\M\mapsto\widehat\M$ defined in \eq{hat}.
\end{proof}

We have two immediate corollaries.

\begin{cor} 
\label{cor:main1}
Let $\lambda \in V(\Q)$ and let $\M\colon\bigcup_{n=1}^\infty I^n\times W_n(\Q) \to I$ be a symmetric and Jensen concave $\Q$-weighted mean on $I$. Then $\M$ satisfies the $\lambda$-weighted Kedlaya inequality \eq{KI}. 

On the other hand, if $\M$ is a symmetric and Jensen convex $\Q$-weighted mean on $I$, then \eq{KI} holds with reversed inequality for all $n\in\N$. 
\end{cor}

Taking the constant sequence $\lambda_n=1$ in the above corollary, we arrive at a statement which was one of the main results of the paper \cite{PalPas16}.

\begin{cor} 
\label{cor:main2}
Let $\M:\bigcup_{n=1}^\infty I^n\to I$ be a symmetric and Jensen concave repetition invariant mean on $I$. Then $\M$ satisfies the discrete Kedlaya inequality \eq{DKI} for all $n\in\N$ and $x\in I^n$.

On the other hand, if $\M$ is a symmetric and Jensen convex repetition invariant mean on $I$, then \eq{DKI} holds with reversed inequality for all $n\in\N$ and $x\in I^n$. 
\end{cor}

In our subsequent result we demonstrate the assumption that $\big(\frac{\lambda_i}{\lambda_1+\cdots+\lambda_i}\big)_{i=1}^{n}$ is nonincreasing is not only a technical condition but, in some sense, it is an unavoidable condition.

\begin{thm} 
\label{thm:monotone}
Let $R\subset\R$ be a subring, $n\geq2$ and let $\lambda \in W_n^0(R)$ be a fixed sequence. Let $\M\colon\bigcup_{m=1}^n [0,\infty)^m\times W_m(R) \to [0,\infty)$ be a homogeneous function with following properties:
\begin{enumerate}[(i)]
 \item $\M((0,\dots,0,1),(\lambda_1,\dots,\lambda_{n-1}))=1$ and $\M((0,\dots,0,1),\lambda)=1$;
 \item the mapping $x\mapsto \mu(x):=\M((0,\dots,0,x,1),\lambda)$ is differentiable at $x=0$ with $\mu'(0)<0$. 
\end{enumerate}
Assume that $\M$ satisfies the $(n,\lambda)$-weighted Kedlaya inequality \eq{KI} with reversed inequality sign. 
Then 
\Eq{Lambda}{
  \frac{\lambda_{n-1}}{\lambda_1+\cdots+\lambda_{n-1}}
  \geq \frac{\lambda_n}{\lambda_1+\cdots+\lambda_n}.
}
\end{thm}

\begin{proof} Substituting $x_1=\dots=x_{n-2}=0$ and $x_{n-1}:=x$, $x_n:=1$ into inequality \eq{KI} with reversed inequality sign, then using property (i) of $\M$, we get
\Eq{*}{
  \M\bigg(\Big(0,\dots,0,\frac{\lambda_{n-1}x}{\Lambda_{n-1}},\frac{\lambda_{n-1}x+\lambda_n}{\Lambda_n}\Big),\lambda\bigg)
  \leq \frac{\lambda_{n-1}x+\lambda_n\M((0,\dots,0,x,1),\lambda)}{\Lambda_n}.
}
Therefore, by the homogeneity of $\M$, 
\Eq{*}{
  (\lambda_{n-1}x+\lambda_n)\M\bigg(\Big(0,\dots,0,\frac{\lambda_{n-1}\Lambda_nx}{\Lambda_{n-1}(\lambda_{n-1}x+\lambda_n)},1\Big),\lambda\bigg)
  \leq \lambda_{n-1}x+\lambda_n\M((0,\dots,0,x,1),\lambda),
}
which, using the notation in (ii), can be rewritten as
\Eq{*}{
  (\lambda_{n-1}x+\lambda_n)\mu\bigg(\frac{\lambda_{n-1}\Lambda_nx}{\Lambda_{n-1}(\lambda_{n-1}x+\lambda_n)}\bigg)
  \leq \lambda_{n-1}x+\lambda_n\mu(x)
}
By the second condition of (i), we have that $\mu(0)=1$. Therefore, subtracting $\lambda_n$ and then dividing by $x$ side by side, we get 
\Eq{*}{
  \lambda_{n-1}\mu\bigg(\frac{\lambda_{n-1}\Lambda_nx}{\Lambda_{n-1}(\lambda_{n-1}x+\lambda_n)}\bigg)
  +\frac{\lambda_n}{x}\bigg(\mu\bigg(\frac{\lambda_{n-1}\Lambda_nx}{\Lambda_{n-1}(\lambda_{n-1}x+\lambda_n)}\bigg)-1\bigg)
  \leq \lambda_{n-1}+\lambda_n\frac{\mu(x)-\mu(0)}x
}
Upon taking the limit $x\to0$, using the differentiability of $\mu$ at $0$, we arrive at
\Eq{*}{
  \lambda_{n-1}+\frac{\lambda_{n-1}\Lambda_n}{\Lambda_{n-1}}\mu'(0)\leq \lambda_{n-1}+\lambda_n\mu'(0).
}
Now, using $\mu'(0)<0$, we obtain that \eq{Lambda} holds true.
\end{proof}

The following result is an immediate consequence of the latter theorem.

\begin{cor} 
\label{cor:monotone}
Let $R\subset\R$ be a subring and let $\lambda \in W^0(R)$ be a fixed sequence. Let $\M\colon\bigcup_{n=1}^\infty [0,\infty)^n\times W_n(R) \to [0,\infty)$ be a homogeneous function with following properties:
\begin{enumerate}[(i)]
 \item for all $n\in\N$, $\M((0,\dots,0,1),(\lambda_1,\dots\lambda_n))=1$;
 \item for all $n\geq2$, the mapping $x\mapsto \mu_n(x):=\M((0,\dots,0,x,1),(\lambda_1,\dots,\lambda_n))$ is differentiable at $x=0$ with $\mu_n'(0)<0$. 
\end{enumerate}
Assume that $\M$ satisfies the $\lambda$-weighted Kedlaya inequality \eq{KI} with reversed inequality sign. Then the sequence $\big(\frac{\lambda_n}{\lambda_1+\cdots+\lambda_n}\big)_{n=1}^\infty$ is nonincreasing, that is $\lambda\in V(R)$.
\end{cor}

\begin{exa}
\label{exa:G12}
In this example we construct a homogeneous Jensen convex symmetric mean $\M$ such that, for a positive sequence $\lambda\in W(\Q)$, the reversed $\lambda$-Kedlaya inequality can hold if and only if the sequence $\big(\frac{\lambda_n}{\lambda_1+\cdots+\lambda_n}\big)_{n=1}^\infty$ is nonincreasing. This shows that the latter condition is not a technical one, but it is indispensable.

Consider the function $\M\colon\bigcup_{n=1}^\infty [0,\infty)^n\times W_n(\R) \to [0,\infty)$ defined by
\Eq{*}{
  \M((x_1,\dots,x_n),(\lambda_1,\dots,\lambda_n))
  :=\begin{cases}
     \frac{\lambda_1x_1^2+\cdots+\lambda_nx_n^2}{\lambda_1x_1+\cdots+\lambda_nx_n} &\mbox{if } \lambda_1x_1+\cdots+\lambda_nx_n>0,\\
     0 &\mbox{if } \lambda_1x_1+\cdots+\lambda_nx_n=0.\\
    \end{cases}
}
We first show that $\M$ is a homogeneous Jensen convex symmetric mean. The homogeneity and symmetry are obvious. For the proof of the Jensen convexity, let $x,y\in[0,\infty)^n$ and $\lambda\in W_n(\R)$. We have to verify that
\Eq{JC}{
   \M\Big(\frac{x+y}{2},\lambda\Big)\leq \frac12\big(\M(x,\lambda)+\M(y,\lambda)\big).
}
If the left hand side is zero, then there is nothing to prove. In the other case, by the definition of the mean, we have that $\lambda_1(x_1+y_1)+\cdots+\lambda_n(x_n+y_n)>0$. If $\lambda_1x_1+\cdots+\lambda_nx_n=0$, then $\M(x,\lambda)=0$ and, for all $i\in\{1,\dots,n\}$, we have that $\lambda_i x_i=0$. Therefore,
\Eq{*}{
  \M\Big(\frac{x+y}{2},\lambda\Big)
  &=\frac{\lambda_1(x_1+y_1)^2+\cdots+\lambda_n(x_n+y_n)^2}{2(\lambda_1(x_1+y_1)+\cdots+\lambda_n(x_n+y_n))}\\
  &=\frac{\lambda_1y_1^2+\cdots+\lambda_ny_n^2}{2(\lambda_1y_1+\cdots+\lambda_ny_n)}
   =\frac12\big(\M(x,\lambda)+\M(y,\lambda)\big).
}
In the other subcase $\lambda_1y_1+\cdots+\lambda_ny_n=0$, a completely analogous argument yields that \eq{JC} is valid, too.
Therefore, in the rest of the proof of the Jensen convexity, we may assume that $\lambda_1x_1+\cdots+\lambda_nx_n>0$ and $\lambda_1y_1+\cdots+\lambda_ny_n>0$. Denote $\M(x,\lambda)$ and $\M(y,\lambda)$ by $u$ and $v$, respectively. 
Then, it follows from the definition of the mean $\M$ that
\Eq{*}{
  \sum_{i=1}^n\lambda_i\Big(\Big(\frac{x_i}{u}\Big)^2-\frac{x_i}{u}\Big)=0, \qquad
  \sum_{i=1}^n\lambda_i\Big(\Big(\frac{y_i}{v}\Big)^2-\frac{y_i}{v}\Big)=0.
}
Now, using the convexity of the function $x\mapsto x^2-x$, we get that
\Eq{*}{
 0&=\frac{u}{u+v}\sum_{i=1}^n\lambda_i\Big(\Big(\frac{x_i}{u}\Big)^2-\frac{x_i}{u}\Big)
   +\frac{v}{u+v}\sum_{i=1}^n\lambda_i\Big(\Big(\frac{y_i}{v}\Big)^2-\frac{y_i}{v}\Big)\\
  &=\sum_{i=1}^n\lambda_i\bigg(\frac{u}{u+v}\Big(\Big(\frac{x_i}{u}\Big)^2-\frac{x_i}{u}\Big)
     +\frac{v}{u+v}\Big(\Big(\frac{y_i}{v}\Big)^2-\frac{y_i}{v}\Big)\bigg)\\
  &\geq \sum_{i=1}^n\lambda_i \Big(\Big(\frac{u}{u+v}\frac{x_i}{u}+\frac{v}{u+v}\frac{y_i}{v}\Big)^2
      -\Big(\frac{u}{u+v}\frac{x_i}{u}+\frac{v}{u+v}\frac{y_i}{v}\Big)\Big)\\  
  &= \sum_{i=1}^n\lambda_i \Big(\Big(\frac{x_i+y_i}{u+v}\Big)^2-\frac{x_i+y_i}{u+v}\Big). 
}
After a simple calculation, this inequality implies that 
\Eq{*}{
  \M\Big(\frac{x+y}{2},\lambda\Big)\leq \frac{u+v}{2},
}
which is equivalent to the inequality \eq{JC}.

Let $\lambda\in W(\Q)$ be any sequence with positive terms. We show that the reversed $\lambda$-Kedlaya inequality \eq{KI} is satisfied by $\M$ if and only if $\lambda\in V(\Q)$.

In view of the symmetry and the Jensen convexity of $\M$, if $\lambda\in V(\Q)$, then, by Theorem~\ref{thm:main}, $\M$ fulfills the reversed $\lambda$-Kedlaya inequality \eq{KI}.

On the other hand, assume that $\M$ satisfies the reversed $\lambda$-Kedlaya inequality \eq{KI}. In order to obtain that $\lambda\in V(\Q)$, by Corollary~\ref{cor:monotone} it suffices to verify that $\M$ satisfies conditions (i) and (ii) of this result. Condition (i) is trivially valid. To see that (ii) also holds, observe that
\Eq{*}{
  \mu_n(x)=\frac{\lambda_{n-1}x^2+\lambda_n}{\lambda_{n-1}x+\lambda_n} \qquad(x\geq0).
}
Then
\Eq{*}{
  \mu'(0)=-\frac{\lambda_{n-1}}{\lambda_n}<0.
}
Thus, by  Corollary~\ref{cor:monotone}, the sequence $\big(\frac{\lambda_n}{\lambda_1+\cdots+\lambda_n}\big)_{n=1}^\infty$ must be nonincreasing, i.e., $\lambda\in V(\Q)$ should be valid.
\end{exa}

At the very end of this section let us emphasize that the Kedlaya property is stable under affine transformations of means. More precisely we can establish the following simple lemma.
\begin{lem}
\label{lem:affine}
Let $I$ be an interval $R$ be a ring, $n\in \N$ and $\lambda\in W_n^0(R)$. Let $a,b\in\R$ with $a\ne0$. If an $R$-weighed mean $\M$ on $I$ satisfies the $(n,\lambda)$-Kedlaya inequality \eq{KI} and $a>0$, then this inequality is also satisfied by the mean $\M_{a,b}\colon \bigcup_{k=1}^\infty (a I+b)^k \times W_k(R)\to aI+b$ defined by 
\Eq{*}{
 \M_{a,b}(x,\mu):=a\cdot \M\Big(\Big(\frac{x_1-b}{a},\dots,\frac{x_k-b}{a}\Big),\mu\Big)+b,\qquad (k\in\N,\, (x,\mu)\in(aI+b)^k \times W_k(R)).
}
If $a<0$ then the sign in the inequality \eq{KI} is reversed.
\end{lem}

We note that, similar invariance property holds concerning Jensen convexity and concavity of means.


From now on, we will extensively use Proposition~\ref{prop:main_R}. To make the notation easier let us define, for every $n \in \N$, 
\begin{enumerate}[(a)]
\item $\WQ_n$ to be the set of all $\lambda \in W_n^0(\Q)$ such that the $(n,\lambda)$-Kedlaya inequality is satisfied for every symmetric and Jensen concave $\Q$-weighted mean;
 \item $\WR_n$ to be the set of all $\lambda \in W_n^0(\R)$ such that the $(n,\lambda)$-Kedlaya inequality is satisfied for every symmetric and Jensen concave $\R$-weighted mean which is continuous in the weights.
\end{enumerate}
It is quite easy to observe that this property does not depend on the selection of the domain (cf. \cite{Pas16a}).
It will be mostly used to distinguish the (technical) assumptions of Theorem~\ref{thm:main} and the requirements to the family of means. In fact requirements of Jensen concavity of mean and its symmetry were taken just to provide assumption of these two results to be satisfied. In fact, each collection of constraints leads us to an analogous family of sets. 

Some properties of these sets are implied just by their definition. For example as an immediate result of continuity in weights, we get that $\WR_n$ is a closed subset of $W_n^0(\R)$. Furthermore, the nullhomogeneity in the weights implies that $\WR_n$ is a cone, that is, $c \lambda \in \WR_n$ for all $c>0$ and $\lambda \in \WR_n$. 

Having this notations already introduced, Theorem~\ref{thm:main} and Example~\ref{exa:G12} imply
\Eq{inclQn}{
V_n(\Q)\subseteq\WQ_n \subseteq 
\Big\{\lambda \in W_n^0(\Q) \mathrel{\Big|} \frac{\lambda_{n-1}}{\lambda_1+\cdots+\lambda_{n-1}}\ge \frac{\lambda_{n}}{\lambda_1+\cdots+\lambda_{n}}\Big\}.
}
But $\WQ_n \subseteq \WR_n$ and $\WR_n$ is closed in $W_n^0(\R)$, therefore we obtain a generalization of Theorem~\ref{thm:Kedlaya99} to a broad family of $\R$-weighted means.

\begin{prop}
\label{prop:main_R}
For every $n \in \N$ with $n\geq2$, the following inclusions are valid
\Eq{inclRn}{
V_n(\R)\subseteq \WR_n \subseteq 
\Big\{\lambda \in W_n^0(\R) \mathrel{\Big|} \frac{\lambda_{n-1}}{\lambda_1+\cdots+\lambda_{n-1}}\ge \frac{\lambda_{n}}{\lambda_1+\cdots+\lambda_{n}}\Big\}.
}
\end{prop}

\begin{proof}
Like in the case of $\WQ_n$, the second inclusion is the consequence of Example~\ref{exa:G12}. We will have to prove the first one only.
Let us keep a notation that whenever sequence $\lambda$ is defined, $\Lambda$ denotes its respective sequence of partial sums. 
Define the sets
\Eq{*}{
\mathcal{A}&:=\big\{ \lambda \in V_n(\R) \mathrel{\big|} \lambda_i>0 \text{ for all }i \text{, }\lambda_1=1 \big\},\\
\mathcal{B}&:=\big\{ x \in (0,1]^n \mathrel{\big|} x_1=1>x_2 \text{ and } x \text{ is nonincreasing}\big\}.
}

As $\lambda_1 = \Lambda_1$ for every $\lambda \in \mathcal{A}$, we can define functions $u \colon \mathcal{A} \to \mathcal{B}$ and $v \colon \mathcal{B} \to (0,1)\times(0,1]^{n-2}$ by $u(\lambda):=(\lambda_i/\Lambda_i)_{i=1}^n$ and $v(x):=(x_{i+1}/x_i)_{i=1}^{n-1}$, respectively.
Then we have (with the usual convention $\prod_{j=1}^0 (\cdot) :=1$)
\Eq{*}{
u^{-1}(x)=\Big(x_i \cdot \prod_{j=2}^{i} \frac1{1-x_j}\Big)_{i=1}^n, \qquad v^{-1}(y)=\Big(\prod_{j=1}^{i-1} y_j\Big)_{i=1}^n.
}
Therefore, both $u$ and $v$ are homeomorphisms. So is $w:=v \circ u \colon \mathcal{A} \to (0,1)\times(0,1]^{n-2}$. 
Moreover, by verifying both inclusions, we can see that $w(\Q^n \cap \mathcal{A}) = 
\Q^{n-1}\cap\big((0,1)\times(0,1]^{n-2}\big)$.

Now take any $\lambda^{(0)}\in V_n(\R)$. If $\lambda^{(0)}_k=0$ for some $k \le n$ then $(n,\lambda)$-Kedlaya inequality 
reduces to $(k-1,\lambda)$-Kedlaya inequality. Therefore we may suppose that all entires of $\lambda^{(0)}$ are 
positive. Equivalently, by the nullhomogeneity with respect to the weights, we may assume that $\lambda^{(0)}_1 =1$, 
i.e., $\lambda^{(0)} \in \mathcal{A}$.

Define $a^{(0)}:=w(\lambda^{(0)}) \in (0,1)\times(0,1]^{n-2}$. Take a sequence $(a^{(k)})_{k=1}^{\infty}$ having all 
elements in $\big(\Q\cap(0,1)\big)^{n-1}$ and convergent to $a^{(0)}$. For $\lambda^{(k)}:=s^{-1}(a^{(k)}) \in \Q^n \cap 
\mathcal{A}$, we immediately obtain $\lambda^{(k)} \to \lambda^{(0)}$.  However, by \eq{inclQn}, we know that $(\Q^n 
\cap \mathcal{A}) \subseteq \WQ_n \subseteq \WR_n$ for all $k \in \N$. Therefore $\lambda^{(k)} \in \WR_n$ for all $k 
\in \N$. Thus, as $\WR_n$ is closed in $W_n^0(\R)$, we get $\lambda^{(0)} \in \WR_n$, too.
\end{proof}

\section{Discussion}
In this section we will apply results already obtained to a important families of means. Each of subsection will 
consist of definition of the family, a characterization of Jensen concavity and, finally, applications of the notation 
of $\WR_n$. Let us stress that to get some particular examples we need to use Proposition~\ref{prop:main_R}. This purely 
technical operation will be however omit just to keep the notation more compact.

\subsection{Deviation means}
Given a function $E \colon I \times I \to \R$ vanishing on the diagonal of $I\times I$, continuous and 
strictly decreasing with respect to the second variable (we will call such a function to be a \emph{deviation function}), we can define a mean $\D_E\colon \bigcup_{n=1}^{\infty} I^n \to I$ in the following manner (cf.\ Dar\'oczy \cite{Dar71b}). For every $n\in\N$, for every vector $x 
=(x_1,\dots,x_n)\in I^n$ and $\lambda=(\lambda_1,\dots,\lambda_n)\in W_n(\R)$, the \emph{weighted deviation mean (or Dar\'oczy mean)} $\D_E(x,\lambda)$ is the unique solution $y$ of the equation
\Eq{*}{
\lambda_1E(x_1,y)+\dots+\lambda_nE(x_n,y)=0.
}
By \cite{Pal82a} deviation means are symmetric weighted mean which is continuous in the weights. The increasingness of a deviation mean $\D_E$ is equivalent to the increasingness of the deviation $E$ in its first variable. All these properties and characterizations are consequences of the general results obtained in a series of papers by Losonczi 
\cite{Los70a,Los71a,Los71b,Los71c,Los73a,Los77} (for Bajraktarevi\'c means and Gini means) and by Dar\'oczy \cite{Dar71b,Dar72b}, Dar\'oczy--Losonczi \cite{DarLos70}, Dar\'oczy--P\'ales \cite{DarPal82,DarPal83} (for deviation means) and by P\'ales \cite{Pal82a,Pal83b,Pal84a,Pal85a,Pal88a,Pal88d,Pal88e} (for deviation and quasi-deviation means).

The only property which requires some calculations is the characterization of the Jensen concavity of a deviation mean.

\begin{lem}
Let $E \colon I \times I \to \R$ be a deviation function which is differentiable with respect to its second variable and $\partial_2E(t,t)$ is nonvanishing for $t\in I$. Then $\D_E$ is Jensen concave if and only if the mapping $E^*\colon I^2\to\R$ defined by
\Eq{E*}{
  E^*(x,t):=-\frac{E(x,t)}{\partial_2E(t,t)}
}
is Jensen concave.
\end{lem}

\begin{proof}
Define $\widehat{E}(x,t):=-E(-x,-t)$ for $(x,t)\in(-I)^2$, and 
\Eq{*}{
 \widehat{E}^{*}(x,t)=-\frac{\widehat{E}(x,t)}{\partial_2\widehat{E}(t,t)}=-\frac{-E(-x,-t)}{\partial_2 E(-t,-t)}=-E^*(-x,-t).
}
In particular $\widehat{E}^*$ is Jensen convex if and only if $E^*$ is Jensen concave. 
 
Furthermore, in view of the identity $\widehat{\D}_E=\D_{\widehat{E}}$, we have that $\D_E$ is Jensen concave if and only if $\D_{\widehat{E}}$ is Jensen convex. Moreover, applying \cite[Theorem 6]{Pal88a} with appropriate substitutions we obtain that $\D_{\widehat{E}}$ is Jensen convex if and only if $\widehat{E}^\ast$ is Jensen convex. 

Finally, binding all equivalences above, one can easily finish the proof.
\end{proof}

Based on the above lemma, it is simple now to formulate a corollary which is important in view of Proposition~\ref{prop:main_R}.

\begin{prop}\label{prop:KedDev}
Let $E \colon I \times I \to \R$ be a deviation function which is differentiable with respect to its second variable such that $\partial_2E(t,t)$ is nonvanishing for $t\in I$ and the mapping $E^*$ defined by \eq{E*} is Jensen concave. Then $\D_E$ satisfies the $(n,\lambda)$-weighted Kedlaya inequality for all $n \in \N$ and $\lambda \in \WR_n$.
\end{prop}

Observe that if $E(x,y)=f(x)-f(y)$ for some continuous, strictly monotone function $f \colon I \to \R$, then the deviation mean $\D_E$ reduces to the quasi-arithmetic mean $\QA{f}$. Therefore, deviation means include quasi-arithmetic means. One can also notice that Bajraktarevi\'c means and Gini means are also form subclasses of deviation means.

\subsection{Homogeneous Deviation means}
It is known \cite{Pal88e} that a deviation mean generated by a continuous deviation function $E\colon \R_+^2\to\R$ is homogeneous if and only if $E$ is of the form $E(x,y)=g(y)f(\tfrac xy)$ for some continuous functions $f,g\colon\R_+\to\R$ such that $f$ vanishes at $1$ and $g$ is positive. Clearly, the deviation mean generated by $E$ is determined only by the function $f$, therefore, as we are going to deal with homogeneous deviation means, let $\E_f$ denote the corresponding deviation mean. 

Let us just mention that homogeneous deviation means generalize power means. Indeed, whenever 
$I=\R_+$ and $f=\pi_p$, where $\pi_p(x):=x^p$ if $p\ne 0$ and $\pi_0(x):=\ln x$, then $\E_{\pi_p}$ coincide with $\P_p$ for all $p \in \R$. It is also known \cite[Theorem~2.3]{PalPas17a}

\begin{thm}\label{CDM}
Let $f\colon\R_+\to\R$ a strictly increasing concave function with $f(1)=0$. Then the function $E\colon\R_+^2\to\R$ defined by $E(x,y):=f\big(\frac xy\big)$ is a deviation and the corresponding deviation mean $\E_f:=\D_E$ is homogeneous, continuous, increasing and Jensen concave.
\end{thm}

This theorem has an immediate corollary which is implied by the definition of $\WR_n$ itself. Its usefulness is provided by Proposition~\ref{prop:main_R}.
\begin{prop}
\label{prop:Kedhomdev}
Let $f\colon\R_+\to\R$ a strictly increasing concave function with $f(1)=0$. Then $\E_f$ satisfies the $(n,\lambda)$-weighted Kedlaya inequality for all $n \in \N$ and $\lambda \in \WR_n$.
\end{prop}

\subsection{Quasi-arithmetic means}
Idea of quasi-arithmetic means first only glimpsed in a pioneering paper by Knopp \cite{Kno28}. Their theory was somewhat later axiomatized in a series of three independent but nearly simultaneous papers by De Finetti \cite{Def31}, Kolmogorov \cite{Kol30}, and Nagumo \cite{Nag30} at the beginning of 1930s. 

Let $I$ be an interval and $f \colon I \to \R$ be a continuous, strictly monotone function. For $n \in \N$ and for a 
given vector $x=(x_1,\dots,x_n)\in I^n$ and $\lambda=(\lambda_1,\dots,\lambda_n)\in W_n(\R)$, set
\Eq{*}{
  \QA{f}(x,\lambda):=f^{-1} \left( \frac{\lambda_1f(x_1)+\cdots+\lambda_nf(x_n)}{\lambda_1+\cdots+\lambda_n} \right).
}
The weighted mean $\QA{f} \colon \bigcup_{n=1}^{\infty} I^n \times W_n(\R)\to I$ defined this way is called the \emph{weighted quasi-arithmetic mean generated by the function $f$}. Quasi-arithmetic means are a natural generalization of power means. Indeed, like in the case of a deviation mean, for all $p \in \R$, means $\QA{\pi_p}$ and $\P_p$ are equal. These means share most of the properties of power means. In particular, it is easy to verify that they are symmetric and strictly increasing. In fact, they admit even more properties of power means (cf. \cite{Kol30}, \cite{Acz48a}). Let us recall some meaningful result \cite[Theorem~2.2]{PalPas17a}

\begin{thm}\label{concQA}
Let $f\colon I \to\R$ be a twice continuously differentiable function with a nonvanishing first derivative. Then the wighted quasi-arithmetic mean $\QA{f}$ is Jensen concave if and only if either $f''$ is identically zero or $f''$ is nowhere zero and the ratio function $\frac{f'}{f''}$ is a convex and negative function on $I$.
\end{thm}

Similarly like it was done in the case of Theorem~\ref{CDM}, this one could be also used to obtain some results concerning Kedlaya inequality. Let us stress again meaningfulness of Proposition~\ref{prop:main_R}.

\begin{prop}\label{prop:KedQA}
Let $f\colon I \to\R$ be a twice continuously differentiable function with a nonvanishing first derivative such that either $f''$ is identically zero or $f''$ is nowhere zero and the ratio function $\frac{f'}{f''}$  is a convex and negative function on $I$. Then $\QA{f}$  satisfies the $(n,\lambda)$-weighted Kedlaya inequality for all $n \in \N$ and $\lambda \in \WR_n$.
\end{prop}

\subsection{Gini means}
Given two real numbers $p,q\in\R$, define the function $\chi_{p,q}\colon\R_+\to\R$ by
\Eq{*}{
  \chi_{p,q}(x)
  :=\begin{cases}
    \dfrac{x^p-x^q}{p-q} & \mbox{ if } p\neq q, \\[4mm]
    x^p\ln(x) & \mbox{ if } p=q.
    \end{cases}
}
In this case, the function $E_{p,q}\colon \R_+^2\to\R$ defined by 
\Eq{*}{
  E_{p,q}(x,y):=y^p\chi_{p,q}\Big(\frac{x}{y}\Big)
}
is a deviation function on $\R_+$. The wighted deviation mean generated by $E_{p,q}$ will be denoted by $\G_{p,q}$ and called the \emph{weighted Gini mean of parameter $p,q$} (cf.\ \cite{Gin38}). One can easily see that $\G_{p,q}$ has the following explicit form:
\Eq{GM}{
  \G_{p,q}(x,\lambda)
   :=\left\{\begin{array}{ll}
    \left(\dfrac{\lambda_1x_1^p+\cdots+\lambda_nx_n^p}
           {\lambda_1x_1^q+\cdots+\lambda_nx_n^q}\right)^{\frac{1}{p-q}} 
      &\mbox{if }p\neq q, \\[4mm]
     \exp\left(\dfrac{\lambda_1x_1^p\ln(x_1)+\cdots+\lambda_nx_n^p\ln(x_n)}
           {\lambda_1x_1^p+\cdots+\lambda_nx_n^p}\right) \quad
      &\mbox{if }p=q.
    \end{array}\right.
}
Clearly, in the particular case $q=0$, the mean $\G_{p,q}$ reduces to the $p$th power mean $\P_p$. It is also obvious that $\G_{p,q}=\G_{q,p}$.
It is known \cite{Los71a,Los71c} that $\G_{p,q}$ is concave if and only if 
\Eq{pq}{
\min(p,q)\leq0\leq\max(p,q)\leq 1.
}
Therefore, as an immediate consequence, we have

\begin{prop}\label{prop:KedGini}
If $p,\,q \in \R$ satisfy \eq{pq}, then $\G_{p,q}$ satisfies the $(n,\lambda)$-weighted Kedlaya inequality for all $n \in \N$ and $\lambda \in \WR_n$.
\end{prop}

\subsection{Power means}
Let just recall from the previous sections that $\P_p=\G_{p,0}=\QA{\pi_p}$, therefore it was already covered in the previous results. In fact we can use either Proposition~\ref{prop:KedQA} or \ref{prop:KedGini} to obtain
\begin{prop}
For every $p\le 1$ the power mean $\P_p$ satisfies the $(n,\lambda)$-weighted Kedlaya inequality for all $n \in \N$ and $\lambda \in \WR_n$.
\end{prop}

Obviously for $p=1$ the power mean $\P_1$ is just an arithmetic mean. Therefore in this case, Kedlaya inequality \eq{KI} becomes an equality for all $n \in \N$ and a pair $x \in \R^n$ with weights $\lambda \in W^0_n(\R)$. In the case $p=0$, the inequality \eq{KI} reduces to the inequality stated in Theorem~\ref{thm:Kedlaya99} which was discovered by Kedlaya \cite{Ked99}. Further important extensions and generalizations of the power mean Kedlaya inequality can be found in the papers \cite{Nan52,Ben99,Bul98,Gao10,Gao15,HarTak03,PecSto01,Sad06,TarTar99,Ciz05,CizPec98}.

\section{Conclusions}

The main result of the paper, the weighted Kedlaya inequality established in Theorem~\ref{thm:main}, generalizes Kedlaya celebrated result of 1999, which was established for the geometric mean. The inequality has several particular cases in the classes of deviation means, quasi-arithmetic means, Gini means and power means.

\bigskip

\subsubsection*{Funding}
The research of the first author was supported by the Hungarian Scientific Research Fund (OTKA) Grant K-111651 and by the
EFOP-3.6.2-16-2017-00015 project. This project is co-financed by the European Union and the European Social Fund.

\subsubsection*{Competing interests}
The authors declare that they have no competing interests.

\subsubsection*{Authors’ contributions}
Both authors contributed equally to the manuscript, read and approved the final manuscript.


\begin{thebibliography}{10}

\bibitem{Acz48a}
J.~Acz\'el.
\newblock {On mean values}.
\newblock {\em Bull. Amer. Math. Soc.}, 54:392–400, 1948.

\bibitem{Ben99}
G.~Bennett.
\newblock {An inequality for {H}ausdorff means}.
\newblock {\em Houston J. Math.}, 25(4):709–744, 1999.

\bibitem{BerDoe15}
F.~Bernstein and G.~Doetsch.
\newblock {Zur {T}heorie der konvexen {F}unktionen}.
\newblock {\em Math. Ann.}, 76(4):514–526, 1915.

\bibitem{BloEsi97}
S.~L. Bloom and Z.~\'Esik.
\newblock {Axiomatizing Shuffle and Concatenation in Languages}.
\newblock {\em Information and Computation}, 139(1):62–91, 1997.

\bibitem{Bul98}
P.~S. Bullen.
\newblock {\em {A {D}ictionary of {I}nequalities}}, volume~97 of {\em {Pitman
  Monographs and Surveys in Pure and Applied Mathematics}}.
\newblock Longman, Harlow, 1998.

\bibitem{Bul03}
P.~S. Bullen.
\newblock {\em {Handbook of means and their inequalities}}, volume 560 of {\em
  {Mathematics and its Applications}}.
\newblock Kluwer Academic Publishers Group, Dordrecht, 2003.

\bibitem{Ciz05}
A.~\v{C}i\v{z}me\v{s}ija.
\newblock {On weighted discrete {H}ardy's inequality for negative power
  numbers}.
\newblock {\em Math. Inequal. Appl.}, 8(2):273–285, 2005.

\bibitem{CizPec98}
A.~\v{C}i\v{z}me\v{s}ija and J.~Pe\v{c}ari\'c.
\newblock {Mixed means and {H}ardy's inequality}.
\newblock {\em Math. Inequal. Appl.}, 1(4):491–506, 1998.

\bibitem{Dar71b}
Z.~Dar\'oczy.
\newblock {A general inequality for means}.
\newblock {\em Aequationes Math.}, 7(1):16–21, 1971.

\bibitem{Dar72b}
Z.~Dar\'oczy.
\newblock {\"Uber eine {K}lasse von {M}ittelwerten}.
\newblock {\em Publ. Math. Debrecen}, 19:211–217 (1973), 1972.

\bibitem{DarLos70}
Z.~Dar\'oczy and L.~Losonczi.
\newblock {\"Uber den {V}ergleich von {M}ittelwerten}.
\newblock {\em Publ. Math. Debrecen}, 17:289–297 (1971), 1970.

\bibitem{DarPal82}
Z.~Dar\'oczy and Zs. P\'ales.
\newblock {On comparison of mean values}.
\newblock {\em Publ. Math. Debrecen}, 29(1-2):107–115, 1982.

\bibitem{DarPal83}
Z.~Dar\'oczy and Zs. P\'ales.
\newblock {Multiplicative mean values and entropies}.
\newblock In {\em {Functions, series, operators, Vol. I, II (Budapest, 1980)}},
  page 343–359. North-Holland, Amsterdam, 1983.

\bibitem{Def31}
B.~de~Finetti.
\newblock {{S}ul concetto di media}.
\newblock {\em Giornale dell' Instituto, Italiano degli Attuarii}, 2:369–396,
  1931.

\bibitem{Gao10}
P.~Gao.
\newblock {A note on mixed-mean inequalities}.
\newblock {\em J. Inequal. Appl.}, pages Art. ID 509323, 8, 2010.

\bibitem{Gao15}
P.~Gao.
\newblock {On a discrete weighted mixed arithmetic-geometric mean inequality}.
\newblock {\em Math. Inequal. Appl.}, 18(3):941–947, 2015.

\bibitem{Gin38}
C.~Gini.
\newblock {{D}i una formula compressiva delle medie}.
\newblock {\em Metron}, 13:3–22, 1938.

\bibitem{HarTak03}
T.~Hara and S.-E. Takahasi.
\newblock {On weighted extensions of {C}arleman's inequality and {H}ardy's inequality}.
\newblock {\em Math. Inequal. Appl.}, 6(4):667–674, 2003.

\bibitem{HarLitPol34}
G.~H. Hardy, J.~E. Littlewood, and G.~P\'olya.
\newblock {\em {Inequalities}}.
\newblock Cambridge University Press, Cambridge, 1934.
\newblock (first edition), 1952 (second edition).

\bibitem{Hol92}
F.~Holland.
\newblock {On a mixed arithmetic-mean, geometric-mean inequality}.
\newblock {\em Math. Competitions}, 5:60–64, 1992.

\bibitem{Ked94}
K.~S. Kedlaya.
\newblock {Proof of a mixed arithmetic-mean, geometric-mean inequality}.
\newblock {\em Amer. Math. Monthly}, 101(4):355–357, 1994.

\bibitem{Ked99}
K.~S. Kedlaya.
\newblock {Notes: {A} {W}eighted {M}ixed-{M}ean {I}nequality}.
\newblock {\em Amer. Math. Monthly}, 106(4):355–358, 1999.

\bibitem{Kno28}
K.~Knopp.
\newblock {\"Uber {R}eihen mit positiven {G}liedern}.
\newblock {\em J. London Math. Soc.}, 3:205–211, 1928.

\bibitem{Kol30}
A.~N. Kolmogorov.
\newblock {{S}ur la notion de la moyenne}.
\newblock {\em Rend. Accad. dei Lincei (6)}, 12:388–391, 1930.

\bibitem{Kuc85}
M.~Kuczma.
\newblock {\em {An {I}ntroduction to the {T}heory of {F}unctional {E}quations
  and {I}nequalities}}, volume 489 of {\em {Prace Naukowe Uniwersytetu
  \'Sl\k{a}skiego w Katowicach}}.
\newblock Państwowe Wydawnictwo Naukowe — Uniwersytet \'Sl\k{a}ski,
  Warszawa–Krak\'ow–Katowice, 1985.
\newblock 2nd edn. (ed. by A. Gil\'anyi), Birkhäuser, Basel, 2009.

\bibitem{Los70a}
L.~Losonczi.
\newblock {\"Uber den {V}ergleich von {M}ittelwerten die mit
  {G}ewichtsfunktionen gebildet sind}.
\newblock {\em Publ. Math. Debrecen}, 17:203–208 (1971), 1970.

\bibitem{Los71a}
L.~Losonczi.
\newblock {Subadditive {M}ittelwerte}.
\newblock {\em Arch. Math. (Basel)}, 22:168–174, 1971.

\bibitem{Los71c}
L.~Losonczi.
\newblock {Subhomogene {M}ittelwerte}.
\newblock {\em Acta Math. Acad. Sci. Hungar.}, 22:187–195, 1971.

\bibitem{Los71b}
L.~Losonczi.
\newblock {\"Uber eine neue {K}lasse von {M}ittelwerten}.
\newblock {\em Acta Sci. Math. (Szeged)}, 32:71–81, 1971.

\bibitem{Los73a}
L.~Losonczi.
\newblock {General inequalities for nonsymmetric means}.
\newblock {\em Aequationes Math.}, 9:221–235, 1973.

\bibitem{Los77}
L.~Losonczi.
\newblock {Inequalities for integral mean values}.
\newblock {\em J. Math. Anal. Appl.}, 61(3):586–606, 1977.

\bibitem{Nag30}
M.~Nagumo.
\newblock {\"Uber eine {K}lasse der {M}ittelwerte}.
\newblock {\em Japanese J. Math.}, 7:71–79, 1930.

\bibitem{Nan52}
T.~S. Nanjundiah.
\newblock {Sharpening of some classical inequalities}.
\newblock {\em Math. Student}, 20:24–25, 1952.

\bibitem{Pas16a}
P.~Pasteczka.
\newblock {Limit properties in a family of quasi-arithmetic means}.
\newblock {\em Aequationes Math.}, 90(4):773–785, 2016.

\bibitem{PecSto01}
J.~E. Pe\v{c}ari\'c and K.~B. Stolarsky.
\newblock {Carleman's inequality: history and new generalizations}.
\newblock {\em Aequationes Math.}, 61(1–2):49–62, 2001.

\bibitem{Pal82a}
Zs. P\'ales.
\newblock {Characterization of quasideviation means}.
\newblock {\em Acta Math. Acad. Sci. Hungar.}, 40(3-4):243–260, 1982.

\bibitem{Pal83b}
Zs. P\'ales.
\newblock {On complementary inequalities}.
\newblock {\em Publ. Math. Debrecen}, 30(1-2):75–88, 1983.

\bibitem{Pal84a}
Zs. P\'ales.
\newblock {Inequalities for comparison of means}.
\newblock In W.~Walter, editor, {\em {General Inequalities, 4 (Oberwolfach,
  1983)}}, volume~71 of {\em {International Series of Numerical Mathematics}},
  page 59–73. Birkh\"auser, Basel, 1984.

\bibitem{Pal85a}
Zs. P\'ales.
\newblock {Ingham {J}essen's inequality for deviation means}.
\newblock {\em Acta Sci. Math. (Szeged)}, 49(1-4):131–142, 1985.

\bibitem{Pal88a}
Zs. P\'ales.
\newblock {General inequalities for quasideviation means}.
\newblock {\em Aequationes Math.}, 36(1):32–56, 1988.

\bibitem{Pal88d}
Zs. P\'ales.
\newblock {On a {P}exider-type functional equation for quasideviation means}.
\newblock {\em Acta Math. Hungar.}, 51(1-2):205–224, 1988.

\bibitem{Pal88e}
Zs. P\'ales.
\newblock {On homogeneous quasideviation means}.
\newblock {\em Aequationes Math.}, 36(2-3):132–152, 1988.

\bibitem{PalPas16}
Zs. P\'ales and P.~Pasteczka.
\newblock {Characterization of the {H}ardy property of means and the best
  {H}ardy constants}.
\newblock {\em Math. Inequal. Appl.}, 19(4):1141–1158, 2016.

\bibitem{PalPas17a}
Zs. P\'ales and P.~Pasteczka.
\newblock {On the best {H}ardy constant for quasi-arithmetic means and
  homogeneous deviation means}.
\newblock {\em Math. Inequal. Appl.}, 21(2):585--599, 2018.

\bibitem{Sad06}
R.~Kh. Sadikova.
\newblock {Comparison of discrete mixed means containing symmetric functions}.
\newblock {\em Math. Notes}, 80:254–260, 2006.

\bibitem{TarTar99}
Ch.~D. Tarnavas and D.~D. Tarnavas.
\newblock {An inequality for mixed power means}.
\newblock {\em Math. Inequal. Appl.}, 2(2):175–181, 1999.

\end{thebibliography}

\end{document}